\documentclass{amsart}
\usepackage[T1]{fontenc}
\usepackage{amsmath,amsthm,amsfonts}
\usepackage{amssymb, amsmath, amsthm, amsgen, amstext, amsbsy, amsopn}
\usepackage{amsfonts}

\newtheorem{thm}{Theorem}[section]
\newtheorem{lemma}[thm]{Lemma}

\newtheorem{cor}[thm]{Corollary}
\theoremstyle{definition}
\newtheorem{defin}[thm]{Definition}
\newtheorem{rem}[thm]{Remark}

\newtheorem{cond}[thm]{Condition}
\newtheorem*{notations}{Notations}
\newtheorem*{acknowledgements}{Acknowledgements}

\def\eps{\epsilon}

\def\part{\partial}

\def\const{{\rm const}}

\newcommand{\hG}{{\widehat{\Gamma}}}

\newcommand{\Lip}{{\operatorname{Lip}}}
\newcommand{\re}{\operatorname{Re}}
\newcommand{\im}{\operatorname{Im}}

\newcommand{\R}{{\mathbb{R}}}

\newcommand{\T}{{\mathbb{T}}}
\newcommand{\Z}{{\mathbb{Z}}}
\newcommand{\N}{{\mathbb{N}}}

\newcommand{\hl}{{\widehat{\ell}}}

\newcommand{\id}{{\text{{\bf 1}}}}



\title{Growth and mixing }

\author[K.\ Fr\k{a}czek]{Krzysztof Fr\k{a}czek}
\address{Faculty of Mathematics and Computer Science,
Nicolaus Copernicus University, Chopina 12/18,  87-100 Toru\'n,
Poland, Institute of Mathematics Polish Academy of Science,
\'Sniadeckich 8, 00-956 Warszawa, Poland}
\email{fraczek@mat.uni.torun.pl}
\thanks{Research partially supported by Marie Curie "Transfer of Knowledge" program,
project MTKD-CT-2005-030042 (TODEQ)}

\author[L.\ Polterovich]{Leonid Polterovich}
\address{School of Mathematical Sciences,
 Tel Aviv University,
 Tel Aviv 69978, Israel}
 \email{polterov@post.tau.ac.il}
\thanks{}

\date{\today}
\begin{document}

\begin{abstract}
Given a bi-Lipschitz measure-preserving homeomorphism of a compact
metric measure space of finite dimension, consider the sequence
formed by the Lipschitz norms of its iterations. We obtain lower
bounds on the growth rate of this sequence assuming that our
homeomorphism mixes a Lipschitz function. In particular, we get a
universal lower bound which depends on the dimension of the space
but not on the rate of mixing. Furthermore, we get a lower bound
on the growth rate in the case of rapid mixing. The latter turns
out to be sharp: the corresponding example is given by a symbolic
dynamical system associated to the Rudin-Shapiro sequence.
\end{abstract}

\maketitle

\section{Introduction and main results} \label{intro}
Let $(M,\rho,\mu)$ be a compact metric space endowed with a
probability Borel measure $\mu$ with $\text{supp}(\mu)=M$. Denote
by $G$ the group of all bi-Lipschitz homeomorphisms of $(M,\rho)$
which preserve the measure $\mu$. For $\phi \in G$ write
$\Gamma(\phi)=\Gamma_\rho(\phi)$ for the maximum of the Lipschitz
constants of $\phi$ and $\phi^{-1}$. Note that $\Gamma(\phi)$ is a
sub-multiplicative: $\Gamma(\phi\psi) \leq
\Gamma(\phi)\cdot\Gamma(\psi)$. Thus $\log \Gamma$ is a
pseudo-norm on $G$, which enables us to consider the group $G$ as
a geometric object. In the present note we discuss a link between
dynamics of $\phi \in G$ (the rate of mixing) and geometry of the
cyclic subgroup of $G$ generated by $\phi$ (the growth rate of
$\Gamma(\phi^n)$ as $n \to \infty$.) On the geometric side, we
focus on the quantity
$$\hG _n(\phi) := \max_{i=1,\ldots,n}
\Gamma(\phi^i)\;.$$

\begin{notations} We write $(f,g)_{L_2}$ for the $L_2$-scalar
product on $L_2(M,\mu)$. We denote by $E$ the space of all
Lipschitz functions on $M$ with zero mean with respect to $\mu$.
We write $||f||_{L_2}$ for the $L_2$-norm of a function $f$,
$\Lip(f)$ for the Lipschitz constant of $f$ and $||f||_{\infty}$
for the uniform norm of $f$.
\end{notations}

\begin{defin}\label{def-mix} We say that a diffeomorphism $\phi \in G$ {\it mixes} a function $f
\in E$ if $(f \circ \phi^n, f)_{L_2} \to 0$ as $n \to \infty$.
\end{defin}

\medskip
\noindent It is known that there exist volume-preserving
diffeomorphisms $\phi$ of certain smooth closed manifolds $M$ with
arbitrarily slow growth of $\hG_n(\phi),  n \to \infty$ (see e.g.
Borichev \cite{Borichev} for $M = \T^2$ and Fuchs \cite{Fuchs} for
extension of Borichev's results to manifolds admitting an effective
$\T^2$-action). As we shall see below, the situation changes if we
assume that $\phi$ mixes a Lipschitz function: in this case the
growth rate of $\hG_n(\phi)$ admits a universal lower bound.
Furthermore the bound becomes better provided the rate of mixing is
decaying sufficiently fast.

To state our first result we need the following invariant of the
metric space $(M,\rho)$. Denote by $E_{R,C}$, where $ R,C \geq 0$,
the subset of functions $f \in E$ with $\Lip(f) \leq R$ and
$\|f\|_{\infty}\leq C$. By the Arzela-Ascoli theorem $E_{R,C}$ is
compact with respect to the uniform norm. Denote by
$D(R,\epsilon,C)$ the minimal number of $\epsilon/2$-balls (in the
uniform norm) needed to cover $E_{R,C}$. Note that for fixed
$\epsilon$ and $C$ the function $D(R,\epsilon,C)$ is non-decreasing
with $R$. For $t\geq D(0,1.4,C)= [C/0.7]+1$ put
$$\tau(t,C) := \sup\{R \geq 0\;:\; D(R, 1.4,C) \leq t\}\;.$$

\begin{thm} \label{thm-mixing} Assume that a bi-Lipschitz homeomorphism
$\phi \in G$ mixes a function $f \in E$ with $||f||_{L_2}=1$. Then
there exists $\alpha>0$ so that $$\hG_n(\phi) \geq
\frac{\tau(\alpha n,\|f\|_{\infty})}{\Lip(f)}$$ for all
sufficiently large $n$.
\end{thm}

\medskip
\noindent The proof is given in Section \ref{sec-proof-mix}.

\medskip
\noindent For a compact subset $A$ of a metric space $(X,\rho_1)$
and $\epsilon>0$ denote by $\mathcal{N}_{\epsilon}(A)$ the minimal
number of open balls with radius $\epsilon/2$ such that their
union covers $A$. Then the upper box dimension of $(A,\rho_1)$ is
defined as
\begin{equation}\label{def-box}
\overline{\dim}_B(A)=\overline{\lim_{\epsilon\to
0}}\frac{\log\mathcal{N}_{\epsilon}(A)}{\log1/\epsilon}.
\end{equation}
 Let $(Y,\rho_2)$ be a compact metric space and let
$\mathcal{D}^A_R(Y)\subset Y^A$ stand for the set of Lipschitz
functions $f:A\to Y$ with $\Lip(f)\leq R$, where $Y^A$ is equipped
with the uniform distance
$$\text{dist}(f,g) = \sup_{x \in A} \rho_2(f(x),g(x))\;.$$
It is easy to show (the proof is analogous to that of Theorem~XXV in
\cite{Kol-Tih})
\begin{equation}\label{Ko-Ti-ineq}\mathcal{N}_{\epsilon}(\mathcal{D}^A_R(Y))\leq
\mathcal{N}_{\epsilon/4}(Y)^{\mathcal{N}_{\epsilon/(4R)}(A)}.
\end{equation}
For the reader's convenience, we present the proof in the
Appendix.

\medskip

Assume now that the metric space $(M,\rho)$ satisfies the following
condition:

\begin{cond}\label{cond-met}
 There exist positive numbers $d$ and $\kappa$ so that for every
$\delta>0$ one can find a $\delta$-net in $(M,\rho)$ consisting of
at most $\kappa \cdot\delta^{-d}$ points.
\end{cond}

\medskip
\noindent This condition is immediately verified if $(M,\rho)$ is a
smooth manifold of dimension $d$ or if $d>\overline{\dim}_B(M)$.
Moreover, it is satisfied for some fractal sets $M\subset\R^n$ where
$d$ is the fractal dimension $M$, e.g.\ if $M$ is a self--similar
set (see Theorem~9.3 \cite{Falconer}).

In what follows $[\alpha]$ denotes the integer part of $\alpha \in
\R$.  Assume that Condition \ref{cond-met} holds. Since
$E_{R,C}=\mathcal{D}^M_R([-C,C])$, by (\ref{Ko-Ti-ineq}), we have
\[D(R,\epsilon,C)\leq
\left(\left[\frac{4C}{\epsilon}\right]+1\right)^{\mathcal{N}_{\epsilon/(4R)}(M)}\\
\leq
\left(\left[\frac{4C}{\epsilon}\right]+1\right)^{\kappa(\epsilon/(4R))^{-d}}.\]
Therefore
$\tau(t,C)\geq \text{const} \cdot\log^{1/d}t.$ Thus Theorem
\ref{thm-mixing} above yields the following:

\begin{cor}\label{cor-log}
If $\phi \in G$ mixes a Lipschitz function then there exists
$\lambda > 0$ so that
$$\hG_n(\phi) \geq \lambda \cdot \log^{\frac{1}{d}} n$$
for all sufficiently large $n$.
\end{cor}

\medskip

This contrasts sharply with the situation when the growth of the
sequence $\Gamma(\phi^n)$ is taken under consideration. In fact, for
every slowly increasing function $u:[0;+\infty)\to[0;+\infty)$ there
exists a volume-preserving real-analytic diffeomorphism of the
$3$--torus which mixes a real-analytic function and such that
$\Gamma(\phi^n)\leq\const\cdot u(n)$ for infinitely many $n$. Such
diffeomorphisms are presented in Section~\ref{slowgrowth}.

\medskip

As a by-product of our proof of Theorem \ref{thm-mixing}  we get the
following result. Let $\phi$ be a bi-Lipschitz homeomorphism of a
compact metric space $M$ satisfying Condition \ref{cond-met}.

\begin{thm} \label{thm-recfin}
If
\begin{equation}
\label{eq-rigid}
\liminf_{n\to\infty}\frac{\hG_n(\phi)}{\log^{1/d}n}=0 \end{equation}
then the cyclic subgroup $\{\phi^n\}$ has the identity map as its
limit point with respect to $C^0$-topology.
\end{thm}

\medskip
\noindent This theorem has the following application to bi-Lipschitz
ergodic theory (the next discussion is stimulated by correspondence
with A.~Katok). Let $T$ be an automorphism of a probability space
$(X,\sigma)$. A {\it bi-Lipschitz realization} of $(X,T,\sigma)$ is
a metric isomorphism between $(X,T,\sigma)$ and $(M,\phi,\mu)$,
where $\phi$ is a bi-Lipschitz homeomorphism of a compact metric
space $M$ equipped with a Borel probability measure $\mu$. An
objective of bi-Lipschitz ergodic theory is to find restrictions on
bi-Lipschitz realizations of various classes of dynamical systems
$(X,T,\sigma)$. The class of interest for us is given by {\it
non-rigid automorphisms} which is defined as follows: Denote by
$U_T$ the induced Koopman operator $f \mapsto f \circ T$ of
$L_2(X,\sigma)$. We say that $T$ is {\it non-rigid} \cite{Katok} if
the closure of the cyclic subgroup generated by $U_T$ with respect
to strong operator topology does {\bf not} contain the identity
operator. Theorem \ref{thm-recfin} shows that {\it any bi-Lipschitz
homeomorphism $\phi$ satisfying condition \eqref{eq-rigid} cannot
serve as a bi-Lipschitz realization of a non-rigid dynamical
system.}

\medskip

Let us return to the study of the interplay between growth and
mixing: Next we explore the influence of the rate of mixing on the
growth of $\hG_n(\phi)$. We shall need the following definitions.

\medskip
\noindent
\begin{defin}\label{def-mix-sp}  Let $\{a_n\}_{n \in \N} $ be a
sequence of positive numbers converging to zero as $n \to \infty$.
We say that a diffeomorphism $\phi \in G$ {\it mixes} a function $f
\in E$ at the rate $\{a_n\}$  if $$|(f \circ \phi^n, f)_{L_2}|\leq
a_n\; \forall n \in \N.$$
\end{defin}

\medskip
\noindent Given a positive sequence $a_n \to 0$, we call a positive
integer sequence $\{v(n)\}$ {\it adjoint} to $\{a_n\}$ if the
following conditions hold:
\begin{equation}\label{eq-adjoint-1}
\sum_{i: 0 < iv(n) \leq n} a_{iv(n)} \leq \frac{1}{4}\;,
\end{equation}
and
\begin{equation}\label{eq-adjoint-2}
\frac{n}{v(n)} \to \infty \;\;\text{as}\;\;n \to \infty\;.
\end{equation}

\begin{lemma}\label{lem-adjoint}
Every positive sequence $a_n \to 0$ admits an adjoint sequence.
\end{lemma}

\medskip
\noindent The proof is given in Section \ref{sec-adjoint}.

\medskip
\noindent In the next theorem we assume that the metric space
$(M,\rho)$ satisfies Condition \ref{cond-met}.
\begin{thm} \label{thm-main}\label{exam-conv}
Assume that a bi-Lipschitz homeomorphism $\phi \in G$ mixes a
Lipschitz function $f\in E$ with $||f||_{L_2}=1$ at the rate
$\{a_n\}$. Then for every adjoint sequence $\{v(n)\}$ of $\{a_n\}$
we have
\begin{equation}\label{eq-main}
\hG_n (\phi) \geq   \frac{1}{2\kappa^{\frac{1}{d}}\Lip(f)} \cdot \Big{[}
\frac{n}{2v(n)}\Big{]}^{1/d} \;\;\forall n \in \N\;.
\end{equation}
In particular, if $\sum a_i < \infty$ then
\begin{equation} \label{eq-power}
\hG_n (\phi) \geq \text{const} \cdot n^{\frac{1}{d}}\;.
\end{equation}
\end{thm}

Note that the second part of the theorem is an immediate
consequence of the first part. Indeed, if $\sum a_i < \infty$ then
the adjoint sequence can be taken constant, $v(n)\equiv v_0$ and
(\ref{eq-main}) implies (\ref{eq-power}).

As we shall show in Section \ref{RS} below the estimate
(\ref{eq-power}) is asymptotically sharp: It is attained for the
shift associated with the Rudin-Shapiro sequence.

\begin{cor} \label{exam-power}
Suppose that $\phi \in G$ mixes a Lipschitz function  at the rate
$\{a_n\}$ such that $a_n=O(1/n^\nu)$, where $0<\nu<1$. Then
$$\hG_n(\phi) \geq \text{const} \cdot n^{\frac{\nu}{d}}.$$
\end{cor}

\begin{proof} If $a_n \leq
c/n^{\nu}$ for some $\nu \in (0;1)$ then one readily checks that
for $C> 0$ large enough  there exists a sequence $\{v(n)\}$
adjoint to $\{a_n\}$ such that $v(n) \leq C \cdot n^{1-\nu}$. Thus
$$\hG_n (\phi)\geq \frac{1}{2\kappa^{\frac{1}{d}}\Lip(f)} \cdot \Big{[}
\frac{n}{2Cn^{1-\nu}}\Big{]}^{1/d} \geq \text{const} \cdot
n^{\frac{\nu}{d}}.$$
\end{proof}

\medskip
\noindent{\sc Organization of the paper:} In Section
\ref{sec-proof-mix} we prove the universal lower growth bound given
in Theorem \ref{thm-mixing} for a bi-Lipschitz homeomorphism which
mixes a Lipschitz function (the case of homeomorphism which mixes an
$L_2$--function is also considered). Furthermore, we prove Theorem
\ref{thm-recfin} asserting that if a bi-Lipschitz homeomorphism
grows sufficiently slow, it must have strong recurrence properties
and in particular must be rigid in the sense of ergodic theory. The
section ends with a discussion on comparison of growth rates in
finitely generated groups and in groups of homeomorphisms. In
Section \ref{sec-proof-main} we prove Theorem \ref{thm-main} which
relates the growth rate to the rate of mixing. For the proof, we
derive an auxiliary fact on "almost orthonormal" sequences of
Lipschitz functions. In Section \ref{holder} we generalize the main
results of the paper to the case of H\"{o}lder observables. In
Section \ref{sec-adjoint} we prove existence of adjoint sequences
used in the formulation of Theorem \ref{thm-main}.

Next we pass to constructing examples. In Section \ref{slowgrowth}
we present an example which emphasizes the difference between the
growth rates of sequences $\hG_n(\phi)$ and $\Gamma(\phi^n)$: We
construct a volume-preserving real-analytic diffeomorphism of the
$3$-torus which mixes a real-analytic function and such that
$\Gamma(\phi^{n_i})$ grows arbitrarily slowly along a suitable
subsequence $n_i \to \infty$. In Section \ref{RS} we show that the
bound in Theorem \ref{thm-main} is sharp: It is attained in the case
of a symbolic dynamical system associated to the Rudin-Shapiro
sequence.

Finally, in Appendix we prove Kolmogorov-Tihomirov type estimate
\eqref{Ko-Ti-ineq}.

\section{Recurrence via Arzela-Ascoli compactness}\label{sec-proof-mix}
\begin{proof}[Proof of Theorem \ref{thm-mixing}] Suppose
that the assertion of the theorem is false. Then, considering a
sequence $\alpha_k=1/k$, $k \in \N$ we get  a sequence $\{n_k\}$ so
that $n_k/k\geq [\|f\|_{\infty}/0.7]+1$ and
$$R_k:= \Lip(f)\cdot\hG_{n_k}(\phi) < \tau(n_k/k,\|f\|_{\infty})\;.$$
This yields
$$D(R_k,1.4,\|f\|_{\infty}) \leq n_k/k < m+1,$$
where $m =[n_k/k]\geq 1$. Consider $m+1$ functions
$$f, f\circ \phi^k,\ldots,f \circ \phi^{mk}\;.$$
Since
$$\Lip(g \circ \psi) \leq \Lip(g) \cdot \Gamma(\psi)
\;\;\forall g \in E,\psi \in G\;,$$  these functions lie in the
subset $E_{R_k,\|f\|_{\infty}} \subset E$. Recall that
$E_{R_k,\|f\|_{\infty}}$ can be covered by
$D(R_k,1.4,\|f\|_{\infty})\leq m$ balls (in the uniform norm) of
the radius $0.7$. By the pigeonhole principle, there is a pair of
functions from our collection lying in the same ball. In other
words for some natural numbers $p > q$ we have $||f\circ
\phi^{pk}-f \circ \phi^{qk}||_{\infty} \leq 1.4$. Put $j=(p-q)k$.
We have
$$||f-f \circ \phi^j||_{L_2}\leq ||f-f \circ \phi^j||_{\infty} \leq
1.4\;.$$ Since $$||f||_{L_2}= ||f \circ \phi^j||_{L_2}=1,$$ we
have
$$(f,f\circ\phi^j)_{L_2}= \frac{1}{2}(||f||_{L_2}^2 +||f \circ
\phi^j||_{L_2}^2-||f-f \circ \phi^j||_{L_2}^2)\geq
\frac{1}{2}(1+1-1.4^2)= 0.02\;.$$ Note that $j\geq k$  and thus
increasing $k$ we get the above inequality for arbitrarily large
values of $j$. This contradicts the assumption that $\phi$ mixes
$f$.
\end{proof}

\medskip

Denote by $H$ the group of all bi-Lipschitz homeomorphisms (not
necessarily measure preserving) of a compact metric space
$(M,\rho)$. An argument similar to the one used in the proof above
shows that if the growth rate of $\hG_n(\phi)$ is sufficiently
slow, the cyclic subgroup $\{\phi^n\}$ generated by $\phi$ has the
identity map as its limit point with respect to $C^0$-topology
(cf.\ a discussion in D'Ambra-Gromov \cite[7.10.C,D]{AG}). Here is
a precise statement. Denote by $\Lambda$ the space of Lipschitz
self-maps of $M$. For $\phi \in \Lambda$ write $\Lip(\psi)$ for
the Lipschitz constant of $\psi$. Equip $\Lambda$ with the
$C^0$-distance
$$\text{dist}(\phi,\psi) = \sup_{x \in M} \rho(\phi(x),\psi(x))\;.$$
Denote by $\Lambda_R$ the subset consisting of all maps $\psi$ from
$\Lambda$ with $\Lip(\psi) \leq R$. This subset is compact with
respect to the metric $\text{dist}$ by the Arzela-Ascoli theorem.
Denote by $\Delta(R,\epsilon)$ the minimal number of
$\epsilon/2$-balls required to cover $\Lambda_R$. For $t\geq
\Delta(0,\epsilon)= \mathcal{N}_{\epsilon}(M)$ put
$$\theta(t,\epsilon) = \sup\{R \geq 0\;:\; \Delta(R,\epsilon) \leq
t\}\;.$$

\medskip
\begin{thm}\label{thm-nonrig}
Let $\phi:M\to M$ be a bi-Lipschitz homeomorphism. Assume that the
identity map is not a limit point with respect to $C^0$-topology for
the cyclic subgroup $\{\phi^n\}$. Then for every sequence
$\epsilon_n \to 0$ there exists $\alpha>0$ so that
\[\widehat{\Gamma}_n(\phi) \geq \theta(\alpha
n,\epsilon_n)\] for all sufficiently large $n$.
\end{thm}

\begin{proof} Suppose that the assertion of the theorem is false.
For every $\alpha=1/k$, $k \in \N$ we can choose
 $n_k >\max(\mathcal{N}_{\epsilon_k}(M),k)$ so that
\[\widehat{\Gamma}_{n_k}(\phi) < \theta(n_k/k,\epsilon_k).\]
Put $m_k = [n_k/k]$ and $R_k = \hG_{n_k}(\phi)$. Since
$R_k<\theta(n_k/k,\epsilon_k)$, we obtain
\[\Delta(R_k,\epsilon_k) \leq n_k/k< m_k+1.\]
Consider $m_k+1$ maps $\id, \phi^k,\ldots,\phi^{km_k}$. They lie in
$\Lambda_{R_k}$. Since $\Delta(R_k,\epsilon_k) \leq m_k$, it follows
that at least two of these maps lie in the same $\epsilon_k/2$-ball
covering of $\Lambda_{R_k}$. Therefore there exist $p > q$ so that
$$ \text{dist}(\phi^{pk},\phi^{qk}) \leq \epsilon_k\;.$$ Put $l_k =
(p-q)k$, and note that $\text{dist}(\phi^{pk},\phi^{qk})=
\text{dist}(\id, \phi^{l_k})$. Thus $ \text{dist}(\id, \phi^{l_k})
\leq \epsilon_k$, and since $k$ divides $l_k$ we have $l_k \to
\infty$. We conclude that $\phi^{l_k}  \to \id\;,$ which contradicts
the fact that the identity map is not a limit point (with respect to
$C^0$-topology) for the sequence $\{\phi^n\}$.
\end{proof}

\medskip
\noindent
\begin{rem}\label{exam-recurr-mfds}   Assume that
the metric space $(M,\rho)$ satisfies Condition \ref{cond-met}
with exponent $d>0$. Since $\Lambda_R=\mathcal{D}^R_M(M)$, by
(\ref{Ko-Ti-ineq}), we have

$$\Delta(R,\epsilon) \leq \mathcal{N}_{\epsilon/4}(M)^{\mathcal{N}_{\epsilon/(4R)}(M)}
\leq (\kappa(\epsilon/4)^{-d})^{\kappa(\epsilon/(4R))^{-d}}.$$ Thus
$$
\theta(t,\epsilon) \geq \text{const}    \frac{\epsilon \cdot
\log^{1/d} t}{\log^{1/d} 1/\epsilon}\;.$$
\end{rem}

\medskip
\noindent
\begin{cor}\label{cor-reccur}
Let $\phi:M\to M$ be a bi-Lipschitz homeomorphism, where $M$
satisfies Condition \ref{cond-met}. Assume that the identity map is
not a limit point with respect to $C^0$-topology for the cyclic
group $\{\phi^n\}$. Let $\{\eta(n)\}$ be a sequence of positive
numbers such that $\eta(n)\to+\infty$ as $n\to+\infty$ and
$\eta(n)=o(\log n)$. Then $\hG_n(\phi) \geq \eta(n)^{1/d}$ for all
sufficiently large $n$.
\end{cor}

\medskip
\begin{proof} An application of Theorem~\ref{thm-nonrig} for
$\epsilon_n=(\eta(n)/\log n)^{\frac{1}{2d}}$ gives the existence
of $\alpha>0$ for which
\begin{eqnarray*}\hG_n(\phi) &\geq& \theta(\alpha
n,\epsilon_n)\geq \text{const}    \frac{\epsilon_n \cdot
\log^{1/d} \alpha n}{\log^{1/d} 1/\epsilon_n}\geq \text{const}
\frac{\left(\frac{\eta(n)}{\log n}\right)^{\frac{1}{2d}} \cdot
\log^{1/d} n}{\log^{1/d} \frac{\log n}{\eta(n)}}\\
&=& \text{const} \frac{\left(\frac{\log
n}{\eta(n)}\right)^{\frac{1}{2d}} }{\log^{1/d} \frac{\log
n}{\eta(n)}}\cdot \eta(n)^{1/d}\geq \eta(n)^{1/d}
\end{eqnarray*}
for all sufficiently large $n$.
\end{proof}

\medskip
\noindent Theorem \ref{thm-recfin} is an immediate consequence of
Corollary \ref{cor-reccur}.

\medskip
\noindent \begin{rem} Consider {\it any} group $H$ equipped with a
pseudo-norm $\ell$: $\ell(h) \geq 0$ all $h \in H$, $\ell (h^{-1})
=\ell (h)$ and $\ell (hg) \leq \ell (h) +\ell (g)$. For an element
$h \in G$ put
$$\hl_n (h) = \max_{i=1,\ldots,n} \ell (h^n)\;.$$
It is instructive to compare possible growth rates of cyclic
subgroups in the following two cases:
\begin{itemize}
\item[{(i)}] $H$ is a finitely generated group, $\ell$ is the word norm;
\item[{(ii)}] $H$ is the group of all bi-Lipschitz homeomorphisms equipped
with the pseudo-norm $\ell =\log \Gamma$.
\end{itemize}

We claim that in the first case, condition
\begin{equation}
\label{eq-pattern} \liminf_{n \to \infty} \frac{ \hl_n(\phi)}{\log
n} = 0
\end{equation}
is equivalent to the fact that $\phi$ is of finite order. Indeed,
assume that $\phi$ satisfies (\ref{eq-pattern}). Denote by $H_R
\subset H$ the ball of radius $R$ centred at $\phi$ in the word
norm. Denote by $K$ the number of elements in the generating set
of $H$. Then the cardinality of $H_R$ does not exceed $K^{R+1}$.
Condition \eqref{eq-pattern} guarantees that there exists $n >0$
arbitrarily large so that $\hl _n(\phi) \leq \log n/(2\log K)$.
Consider $n+1$ elements $\id,\phi,\ldots,\phi^n$. All of them lie
in the set $H_R$ with $R = \hl_n(\phi)$. This set contains at most
$K^{R+1} \leq K\sqrt{n}$ elements. Since $K\sqrt{n} < n+1$ for
large $n$ we get that among $\id,\phi,\ldots,\phi^n$ there are at
least two equal elements, hence $\phi$ is of finite order. The
claim follows.

In contrast to this, in the case (ii), the group of bi-Lipschitz
homeomorphisms may have elements of infinite order which satisfy
\eqref{eq-pattern}, see \cite{Borichev, Fuchs}. These elements are
"exotic" from the algebraic viewpoint: they cannot be included
into any finitely generated subgroup $H'$ of $H$ so that the
inclusion $$(H',\text{word\;norm}) \hookrightarrow (H,\log
\Gamma)$$ is quasi-isometric. It would be interesting to explore
more thoroughly the dynamics of these exotic elements.

Corollary \ref{cor-reccur}  shows that if such an exotic element is
of a "very slow" growth then it has strong recurrence properties.
The argument based on the Arzela compactness, which was used in its
proof, imitates the argument showing that condition
\eqref{eq-pattern} characterizes elements of finite order in
finitely generated groups. Let us compare these results for
bi-Lipschitz homeomorphisms of $d$-dimensional spaces. Consider such
a homeomorphism, say, $\phi$ with $\hl_n(\phi)=o(\log n)$, which
means that it is algebraically exotic in the sense of the discussion
above. If $\phi$ satisfies a stronger inequality $\hl_n(\phi) \leq
(\frac{1}{d}-\epsilon) \log \log n$, it is strongly recurrent by
Corollary \ref{cor-reccur} above. We conclude this discussion with
the following open problem: explore dynamical properties of those
bi-Lipschitz homeomorphisms of $d$-dimensional spaces whose growth
sequence $\hl_{n}(\phi)$ falls into the gap between $\frac{1}{d}\log
\log n$ and $o(\log n)$.
\end{rem}

\section{Almost orthonormal systems of Lipschitz functions}\label{sec-proof-main}
In this section we prove Theorem \ref{thm-main}. We start with the
following general result on "almost orthonormal" systems of
functions:

\begin{thm}\label{thm-basis} Let $\{f_i\}$ be a sequence of linear independent
Lipschitz functions from $E$ with $||f_i||_{L_2}=1$  with the
following property: There exists a sequence of  positive real
numbers $a_n \to 0$ so that $|(f_i,f_j)_{L_2}| \leq a_{i-j}$ for
all $j < i$. Let $\{v(n)\}$ be an adjoint sequence of $\{a_n\}$.
Then
\begin{equation}\label{eq-fns}
\max _{i =1,\ldots,n} \Lip(f_i) \geq
\frac{1}{2\kappa^{\frac{1}{d}}} \cdot \Big{[}
\frac{n}{2v(n)}\Big{]}^{1/d} \;\;\forall n \in \N\;.
\end{equation}
\end{thm}

\medskip
\noindent \begin{lemma}\label{lem-fns} Let $f_i \in L_2(M)$, $i =
1,\ldots,N$ be a sequence of functions with $||f_i||_{L_2} =1$ for
all $i$ and $|(f_i,f_j)_{L_2}| \leq \alpha_{i-j}$ for $j < i$, where
$\sum_{i=1}^N \alpha_i \leq 1/4$. Then for every real numbers
$c_1,\ldots,c_N$ we have
$$||\sum_{i=1}^N c_if_i||_{L_2}^2 \geq \frac{1}{2}\sum_{i=1}^N
c_i^2\;.$$
\end{lemma}

\begin{proof} Put
$$h=\sum_{i=1}^N c_if_i\;\text{ and }\; C = \sqrt{\sum_{i=1}^N
c_i^2}\;.$$ Then
$$||h||_{L_2}^2= C^2 +I\;,$$
where $I =2\sum_{j<i} c_ic_j (f_i,f_j)$. By the Cauchy-Schwarz
inequality,
$$|I| \leq 2\sum_{p=1}^N \sum_{j=1}^{N-p} |c_j|\cdot|c_{j+p}|
\cdot \alpha_p \leq 2\cdot \frac{1}{4} \cdot C^2 = C^2/2\;.$$ Thus
$$||h||_{L_2}^2 \geq C^2-C^2/2=C^2/2$$
as required.
\end{proof}

\medskip
\begin{proof}[Proof of Theorem \ref{thm-basis}] We shall assume that
$2v(n) \leq n$, otherwise the inequality (\ref{eq-fns}) holds by
trivial reasons. Put $q(n) =[n/(2v(n))]$ and
$\delta=(\kappa/q(n))^{1/d}$. By the definition of $\kappa$ and
$d$, there exists a $\delta$-net on $M$ consisting of $p \leq
q(n)$ points. Denote by $E' \subset E$ the codimension $p$
subspace consisting of all those functions which vanish at the
points of the net.

Let $V$ be the linear span of the functions $f_{iv(n)}$,
$i=1,\ldots,2q(n)$. Then the dimension of $W:=V \cap E'$ is $\geq
2q(n)-p \geq q(n)$. It is well known \cite[p.103]{Chavel} that
there exists $h \in W$ with
\begin{equation}\label{eq-h-1}
||h||_{\infty} \geq \sqrt{\dim W} ||h||_{L_2}\;.
\end{equation}
Write $h = \sum_{i=1}^{2q(n)} c_i f_{iv(n)}.$ Note that
$|(f_{iv(n)},f_{jv(n)})_{L_2}|\leq a_{(i-j)v(n)}$ for $i < j$. Put
$\alpha_i = a_{iv(n)}$. By the definition of $v(n)$, we have
$$\sum_{i=1}^{2q(n)} \alpha_i \leq \frac{1}{4}\;,$$
and hence by Lemma \ref{lem-fns}
$$||h||_{L_2}^2 \geq C^2/2\;,\;\text{with}\;\; C =
\sqrt{\sum_{i=1}^{2q(n)} c_i^2}\;.$$ We conclude from
\eqref{eq-h-1} that
$$
||h||_{\infty} \geq \frac{1}{\sqrt 2}\cdot \sqrt{q(n)} \cdot C\;.
$$
Recall now that $h$ vanishes at all the points of the $\delta$-net.
Thus
\begin{equation}
\label{eq-h-2} \Lip(h) \geq ||h||_{\infty}/\delta \geq \frac{1}{\sqrt 2}\cdot
\sqrt{q(n)} \cdot C\cdot (\kappa/q(n))^{-1/d}\;.
\end{equation}

Next, let us estimate $\Lip(h) $ from above. Put $$\Pi_n:= \max
_{i =1,\ldots,n} \Lip(f_i)\;.$$ We have
$$\Lip(h) = \Lip\big{(}\sum_{i=1}^{2q(n)} c_i f_{iv(n)}\big{)}
\leq \Pi_n \cdot \sqrt{2q(n)} \cdot C \;.$$ Combining this
inequality with lower bound \eqref{eq-h-2} we get
$$\Pi_n \geq \frac{1}{2\kappa^{\frac{1}{d}}}
\cdot q(n)^{\frac{1}{d}}\;,$$ as required.
\end{proof}

\medskip
\noindent {\sc Reduction of Theorem \ref{thm-main} to Theorem
\ref{thm-basis}:} We start with the following auxiliary lemma.

\medskip
\noindent
\begin{lemma} \label{lem-indep}
Assume that $\phi \in G$ mixes a function $f \in E$. Then for
every $m>0$ the functions $f,f\circ \phi,\ldots,f\circ \phi^m$ are
linearly independent elements of $E$.
\end{lemma}

\begin{proof} Assume that $\|f\|_{L_2}=1$ and on the contrary that for
some $m$ these functions are linearly dependent. Then for some $p
\in \N$
$$f \circ \phi^p \in V:= \text{Span}(f,f\circ \phi,\ldots,f \circ
\phi^{p-1})$$ which implies that {\it every} function of the form $f
\circ \phi^n, n \in \Z$ belongs to $V$. The space $V$ is
finite-dimensional and every element of the sequence $\{f \circ
\phi^n\}, n \in \Z$ has unit $L_2$-norm. Thus this sequence has a
subsequence converging to an element $g\in V$ of unit $L_2$-norm.
Since $\phi$ mixes $f$, we have $(g,f\circ\phi^n)_{L_2}=0$ for every
$n\in\Z$. It follows that $g=0$, contrary to $\|g\|_{L_2}=1$. This
completes the proof.
\end{proof}

\medskip
\begin{proof}[Proof of Theorem \ref{thm-main}] Put $f_i = f\circ
\phi^{i}, i\in \N$. Since $\phi$ mixes $f$ at the rate $\{a_i\}$
we have $|(f_i,f_j)_{L_2}| \leq a_{i-j}$ for all $j < i$. The
functions $\{f_i\}$ are linearly independent by Lemma
\ref{lem-indep}. Thus all the assumptions of Theorem
\ref{thm-basis} hold. Theorem \ref{thm-main} readily follows from
Theorem \ref{thm-basis} combined with the inequality
$$\max _{i =1,\ldots,n} \Lip(f_i) \leq \hG_n (\phi) \cdot
\Lip (f)\;.$$
\end{proof}

\medskip
\noindent
\begin{rem}\label{rem-Weyl} Assume that $\{f_i\}$ is an orthonormal
system (in the $L_2$-sense) of Lipschitz functions with zero mean.
Put
$$\Pi_n:= \max _{i =1,\ldots,n} \Lip(f_i)\;.$$ It follows from Theorem
\ref{thm-basis} that
$$\Pi_n \geq \text{const} \cdot n^{\frac{1}{d}}\;.$$
For an illustration, consider the Euclidean torus $\T^d =
\R^d/\Z^d$. Let $\lambda_1 \leq \lambda_2 \leq \ldots $ be the
sequence of the eigenvalues (taken with their multiplicities) of
the Laplace operator. Each $\lambda_n$ has the form $4\pi^2|v|^2$,
where $v$ runs over $\Z^d \setminus \{0\}$. Choose the sequence of
eigenfunctions $f_n$ corresponding to $\lambda_i$ so that the
eigenfunctions corresponding to $4\pi^2|v|^2$ are $\sqrt{2}\sin
2\pi (x,v)$ and $\sqrt{2}\cos 2\pi (x,v)$. It follows that
$$\Lip(f_n) \approx |v| \approx \lambda_n^{1/2} \approx n^{1/d}\;,$$
where the last asymptotic (up to a multiplicative constant) is
just the Weyl law. It follows that  the exponent of the power-law
in the right hand side of the inequality \eqref{eq-fns} is sharp.
\end{rem}

\section{From Lipschitz to  H\"older observables}\label{holder}
Assume that the metric space $(M,\rho)$ satisfies Condition
\ref{cond-met} with exponent $d>0$. Let $\phi:(M,\rho,\mu)\to
(M,\rho,\mu)$  be a bi-Lipschitz homeomorphism. Suppose that
$f:M\to\R$ is a H\"older continuous function with exponent
$\beta\in(0; 1]$ which is mixed by $\phi$. Let $\rho_\beta$ stand
for the metric on $M$ given by $\rho_\beta(x,y)=\rho(x,y)^\beta$.
Under the new metric $f$ becomes a Lipschitz function and $\phi$
remains bi-Lipschitz with
$\Gamma_{\rho_\beta}(\phi)=\Gamma(\phi)^\beta$. Moreover the
metric space $(M,\rho_\beta)$ satisfies Condition \ref{cond-met}
with exponent $d/\beta$. By Corollary~\ref{cor-log}, we have
\[\hG_n(\phi)^\beta=\widehat{(\Gamma_{\rho_\beta})}_n(\phi) \geq \const \cdot \log^{\frac{\beta}{d}}
n\] which yields the following:

\begin{cor}
If $\phi \in G$ mixes a H\"older continuous function then there
exists $\lambda
> 0$ so that
$$\hG_n(\phi) \geq \lambda \cdot \log^{\frac{1}{d}} n$$
for all natural $n$.
\end{cor}

In the same manner an application of Theorem~\ref{exam-conv} and
Corollary~\ref{exam-power} gives the following:

\begin{cor}
Suppose that $\phi \in G$ mixes a H\"older continuous function  at
the rate $\{a_n\}$ such that $\sum a_n<\infty$. Then there exists
$\lambda
> 0$ so that
$$\hG_n(\phi) \geq \lambda \cdot n^{\frac{1}{d}}$$
for all natural $n$. If $a_n=O(1/n^\nu)$, where $0<\nu<1$ then
there exists $\lambda > 0$ so that
$$\hG_n(\phi) \geq \lambda \cdot n^{\frac{\nu}{d}}$$
for all natural $n$.
\end{cor}

\section{Existence of an adjoint sequence}
\label{sec-adjoint}
\begin{proof}[Proof of Lemma \ref{lem-adjoint}] Making a rescaling if
necessary assume that $a_n \leq 1$ for all $n$. Choose $N_k
\nearrow \infty, k \in \N$ so that $N_1 =1$ and $a_i \leq 1/k$ for
all $i \geq N_k$. Put $b_n := 1/k$ for $n \in [N_k;N_{k+1})$. Thus
$\{b_n\}$ a non-increasing positive sequence which majorates
$\{a_n\}$ and converges to zero.

Define $v(n)$ as {\it the minimal} integer $k$ with
$$\frac{b_k}{k} < \frac{1}{4n}\;.$$  Note that $v(n) \to \infty $ as $n \to \infty$. By
definition
$$\frac{b_{v(n)-1}}{v(n)-1}  \geq
\frac{1}{4n}\;.$$ Thus we get that
$$\frac{v(n)}{4n} \leq b_{v(n)-1} + \frac{1}{4n}\;,$$
and hence $v(n)/n \to 0$ which yields assumption
\eqref{eq-adjoint-2}. Furthermore, using monotonicity of $b_n$ and
inequality $b_{v(n)}/v(n) < 1/(4n)$ which follows from the
definition of $v(n)$ we estimate
$$\sum_{i: 0 < iv(n) \leq n} b_{iv(n)} \leq \frac{n}{v(n)}
\cdot b_{v(n)} \leq \frac14\,$$ and we get assumption
\eqref{eq-adjoint-1}.
\end{proof}

\bigskip

\section{Slowly growing diffeomorphisms} \label{slowgrowth}
As we have shown above, if a bi-Lipschitz homeomorphism $\phi$ of a
$d$-dimensional compact metric space mixes a Lipschitz function, the
growth rate of the sequence $\widehat{\Gamma}_n(\phi)$ is at least
$\sim \log^{1/d} n$ (see Corollary \ref{cor-log}). Furthermore,
$\widehat{\Gamma}_n(\phi)\geq \text{const}\cdot n^{\nu/d}$ provided
the mixing rate is $\sim n^{-\nu}$ for some $\nu \in (0;1)$ (see
Corollary \ref{exam-power} ). In this section we work out an example
which shows that the behavior of the sequence $\Gamma(\phi^n)$ is
essentially different from the one of $\widehat{\Gamma}_n(\phi)$
even in real-analytic category. In addition, this example gives us
an opportunity to test our lower bounds on
$\widehat{\Gamma}_n(\phi)$ in terms of the rate of mixing.

Consider the three dimensional torus $\T^3 = \R^3/\Z^3$ equipped
with the Euclidean metric and the Lebesgue measure. Fix any concave
increasing function $u: [0;+\infty) \to [0;+\infty)$ such that
\[\lim_{x\to+\infty}u(x)=+\infty,\; u(1) \geq 1 \text{ and }u(x)\leq x^{3/4}.\]

\begin{thm}\label{thm-main-slow}
There exists a real-analytic measure-preserving diffeomorphism
$\phi:\T^3 \to \T^3$ with the following properties:
\begin{itemize}
\item[{(i)}] $\phi$ mixes  a nonzero real-analytic function at the rate $\{\log
u(n)/u(n)^{1/3}\}$;
\item[{(ii)}] There exists a positive constant $c_1>0$ such that
$\Gamma(\phi^n)\leq c_1 u(n)$ for infinitely many $n\in\N$;
\item[{(iii)}] There exist positive constants $c_2,c_3$ such that
\[c_2\frac{\sqrt{n}}{\log u(n)}\leq\hG_n(\phi)\leq c_3u(\sqrt{n})\sqrt{n},\]
where the left hand side inequality holds for every natural $n$ and
the right hand side holds for infinitely many $n$.
\end{itemize}
\end{thm}

\medskip
\noindent In particular, this theorem shows that $\Gamma(\phi^n)$
can grow arbitrarily slowly along a subsequence even when $\phi$
mixes a real-analytic function.

\begin{rem}
Taking $u(x)=x^{3\nu}$, for $0<\nu<1/4$, we get a diffeomorphism
$\phi$ which mixes a real-analytic function at the rate
$1/n^{\nu-\eps}$ (for arbitrary small $\eps>0$) and such that
$\hG_n(\phi)\geq\text{const}\cdot n^{1/2-\eps}$. Notice that
applying Corollary \ref{exam-power} we get
$\hG_n(\phi)\geq\text{const}\cdot n^{\nu/3-\eps}$. Thus Corollary
\ref{exam-power} gives a correct prediction of the appearance of a
power law in the lower bound for $\hG_n(\phi)$, though with a
non-optimal exponent. It is an interesting open problem to find the
sharp value of the exponent in Corollary \ref{exam-power}.
\end{rem}

\medskip
\noindent Our construction of a diffeomorphism $\phi$ in Theorem
\ref{thm-main-slow} and the estimate of the rate of mixing follows
the work of Fayad \cite{Fayad} (see also \cite{Fayad2}). The main
additional difficulty in our situation is due to the fact that we
have to keep track of the growth of the differential.

\medskip
\noindent {\sc Preliminaries:} We denote by $\T$ the circle group
$\R/\Z$ which we will constantly identify with the interval
$[0;1)$ with addition mod $1$. For a real number $t$ denote  by
$\|t\|$ its distance to the nearest integer number. For an
irrational $\alpha\in\T$ denote by $\{q_n\}$ its sequence of
denominators, i.e.\
\[q_0=1,\;\; q_1=a_1, \;\; q_{n+1}=a_{n+1}q_n+q_{n-1},\]
where $[0;a_1,a_2,\dots]$ is the continued fraction expansion of
$\alpha$. Then
\begin{equation}\label{ulla}
\frac{1}{2q_{n+1}}<\|q_n\alpha\|<\frac{1}{q_{n+1}}\;\;\;\;\text{
for each natural }n.
\end{equation}
Let $T:\T\to\T$ stand for the corresponding ergodic rotation
$Tx=x+\alpha$. Every measurable function $\varphi:\T\rightarrow
\R$ determines the measurable cocycle over the rotation $T$ given
by
\[\varphi^{(n)}(x)=\left\{\begin{array}{ccc}
\varphi(x)+\varphi(Tx)+\ldots+\varphi(T^{n-1}x) & \mbox{if} & n>0\\
0 & \mbox{if} & n=0\\
-\left(\varphi(T^nx)+\ldots+\varphi(T^{-1}x)\right) & \mbox{if} &
n<0.
\end{array}\right.\]
If $\varphi:\T\to\R$ is a continuous function then
\[\|\varphi^{(m+n)}\|_{\infty}\leq\|\varphi^{(m)}\|_{\infty}+\|\varphi^{(n)}\|_{\infty}\text{ and }
\|\varphi^{(-n)}\|_{\infty}=\|\varphi^{(n)}\|_{\infty}\] for all
integer $m,n$.

 Recall that
\begin{equation}\label{1}
4\|x\|\leq| e^{2\pi ix}-1|\leq 2\pi\|x\|\;\;\;\; \text{ for each
real }x.
\end{equation}

\medskip
\noindent{\sc The construction:}  Let us consider a pair of
irrational numbers $(\alpha,\alpha')$ such that the sequences of
denominators $\{q_n\}$, $\{q_n'\}$ of convergents for their
continued fraction expansion satisfy
\begin{equation}\label{zalalpha}
2u^{-1}(e^{q'_{n-1}})\leq \frac{q_n}{u(q_n)}\leq 3u^{-1}(e^{
q'_{n-1}}),\;\;2u^{-1}(e^{ q_{n}})\leq \frac{q'_n}{u(q'_n)}\leq
3u^{-1}(e^{q_{n}})
\end{equation} for any $n\geq n_0(\alpha,\alpha')$.  Here $n_0$ is a
sufficiently large positive integer which will be chosen in the
course of the proof.  For a given pair we consider real analytic
functions $\varphi, \psi$ on $\T$ given by
\begin{equation}\label{deffunk}
\varphi(x)=\sum_{n=n_0}^\infty\frac{\cos 2\pi q_{n}x}{2\pi
q_{n}u^{-1}(e^{q_{n}})},\;\; \psi(y)=\sum_{n=n_0}^\infty\frac{\cos
2\pi q'_{n}y}{2\pi q'_{n}u^{-1}(e^{q'_{n}})}.
\end{equation}
Let us consider the volume--preserving diffeomorphism
$\phi:\T^3\to\T^3$ given by
\[\phi(x,y,z)=(x+\alpha,y+\alpha',z+\varphi(x)+\psi(y)).\]
We claim that $\phi$ has all the properties listed in Theorem
\ref{thm-main-slow}.

\medskip
\noindent{\sc Starting growth estimates:} Then for each integer $n$
we have
\[\phi^n(x,y,z)=(x+n\alpha,y+n\alpha',z+\varphi^{(n)}(x)+\psi^{(n)}(y))\]
and hence
$\Gamma(\phi^n)\sim\max(\|\varphi'^{(n)}\|_{\infty},\|\psi'^{(n)}\|_{\infty})$.

\begin{lemma} \label{nierpod}
For every $x,y\in\T$ and $k\in\N$ we have
\[|\varphi'^{(q_k)}(x)|\leq\frac{6q_k}{u^{-1}(e^{q_k})},\;\;
|\varphi''^{(q_k)}(x)|\leq\frac{6q^2_k}{u^{-1}(e^{q_k})},\]
\[|\psi'^{(q'_k)}(y)|\leq\frac{48q'_k}{u^{-1}(e^{q'_k})},\;\;
|\psi''^{(q'_k)}(y)|\leq\frac{48q'^2_k}{u^{-1}(e^{q'_k})}.\]
\end{lemma}

\begin{proof} Since
\[\varphi^{(m)}(x)=\sum_{n=n_0}^\infty\frac{1}{2\pi q_{n}u^{-1}(e^{q_{n}})}\re e^{2\pi i q_{n}x}
\frac{e^{2\pi imq_{n}\alpha}-1}{e^{2\pi iq_{n}\alpha}-1},\] we
obtain
\begin{equation}\label{kocpoch}
\varphi'^{(m)}(x)=\sum_{n=n_0}^\infty\frac{1}{u^{-1}(e^{q_{n}})}\im
e^{2\pi i q_{n}x} \frac{e^{2\pi imq_{n}\alpha}-1}{e^{2\pi
iq_{n}\alpha}-1},
\end{equation}
hence
\[|\varphi'^{(q_k)}(x)|\leq\sum_{n=n_0}^\infty
\frac{1}{u^{-1}(e^{q_{n}})}\frac{|e^{2\pi
iq_kq_{n}\alpha}-1|}{|e^{2\pi iq_{n}\alpha}-1|}.\] In the next chain
of inequalities we use that by increasing $n_0$ we can assume that
$\sum_{n=n_0}^{\infty} {q_n}/{u^{-1}(e^{q_{n}})} < 1/4$. We have
\begin{eqnarray*}
\lefteqn{\sum_{n=n_0}^{k-1}
\frac{1}{u^{-1}(e^{q_{n}})}\frac{|e^{2\pi
iq_kq_{n}\alpha}-1|}{|e^{2\pi iq_{n}\alpha}-1|}}\\&\leq &
\sum_{n=n_0}^{k-1}
\frac{2}{u^{-1}(e^{q_{n}})}\frac{\|q_kq_{n}\alpha\|}{\|q_{n}\alpha\|}\leq\sum_{n=n_0}^{k-1}
\frac{2}{u^{-1}(e^{q_{n}})}\frac{q_{n}\|q_k\alpha\|}{\|q_{n}\alpha\|}
\\&\leq&  \sum_{n=n_0}^{k-1}
\frac{4}{u^{-1}(e^{q_{n}})}\frac{q_{n}q_{n+1}}{q_{k+1}}\leq
\frac{4q_{k}}{q_{k+1}}\sum_{n=n_0}^{k-1}
\frac{q_n}{u^{-1}(e^{q_{n}})}\leq\frac{q_{k}}{q_{k+1}} \leq
\frac{q_{k}}{q'_{k}}\;.
\end{eqnarray*}
In view of (\ref{zalalpha}),
\[\frac{q_{k}}{q'_{k}} \leq
\frac{1}{2u(q'_k)}\frac{q_{k}}{u^{-1}(e^{q_{k}})}\leq
\frac{q_{k}}{u^{-1}(e^{q_{k}})}\;.\]
It follows that
\[\sum_{n=n_0}^{k-1}
\frac{1}{u^{-1}(e^{q_{n}})}\frac{|e^{2\pi
iq_kq_{n}\alpha}-1|}{|e^{2\pi iq_{n}\alpha}-1|}\leq
\frac{q_{k}}{u^{-1}(e^{q_{k}})}\;.\]
Furthermore,
\[\sum_{n=k}^\infty \frac{1}{u^{-1}(e^{q_{n}})}\frac{|e^{2\pi
iq_kq_{n}\alpha}-1|}{|e^{2\pi iq_{n}\alpha}-1|}\leq
\sum_{n=k}^\infty \frac{2q_k}{u^{-1}(e^{q_{n}})}\leq
\frac{4q_k}{u^{-1}(e^{q_{k}})},\]  and the required upper bound
for $|\varphi'^{(q_k)}(x)|$ follows.

 Since
\[|\varphi''^{(q_k)}(x)|\leq\sum_{n=n_0}^\infty \frac{2\pi q_n}{u^{-1}(e^{q_{n}})}\frac{|e^{2\pi
iq_kq_{n}\alpha}-1|}{|e^{2\pi iq_{n}\alpha}-1|},\] similar arguments
to those above show that $|\varphi''^{(q_k)}(x)|\leq 48
q_n^2/u^{-1}(e^{q_{k}})$.

The remaining inequalities are proved similarly.
\end{proof}

\begin{lemma}\label{wnpoch}
For every natural $m$ and $k$ we have
\[\|\varphi'^{(m)}\|_{\infty}\leq\frac{6m}{u^{-1}(e^{q_{k}})}+q_k,\;\;
\|\varphi''^{(m)}\|_{\infty}\leq\frac{48mq_k}{u^{-1}(e^{q_{k}})}+q_k,\]
\[\|\psi'^{(m)}\|_{\infty}\leq\frac{6m}{u^{-1}(e^{q'_{k}})}+q'_k,\;\;
\|\psi''^{(m)}\|_{\infty}\leq\frac{48mq'_k}{u^{-1}(e^{q'_{k}})}+q'_k.\]
\end{lemma}

\begin{proof} Write $m$ as $m=pq_k+r$, where $p=[m/q_k]$ and $0\leq
r<q_k$. Then
\[\|\varphi'^{(m)}\|_{\infty}\leq p\|\varphi'^{(q_k)}\|_{\infty}+\|\varphi'^{(r)}\|_{\infty}
\leq
\frac{m}{q_k}\frac{6q_k}{u^{-1}(e^{q_{k}})}+r\|\varphi'\|_{\infty}\leq\frac{6m}{u^{-1}(e^{q_{k}})}+q_k.
\]
The remaining inequalities are proved similarly.
\end{proof}

\noindent{\sc A van der Corput like Lemma:\footnote{It is known
also as a stationary phase argument.}} For estimating the rate of
mixing, we shall need the following version of the van der Corput
Lemma:

\medskip
\begin{lemma}\label{lemfund}
Let $f:\T\rightarrow{\mathbb{R}}$ be a $C^1$ function. Suppose
there exist a family $\{(a_j;b_j)\subset\T:j=1,\ldots,s\}$ of
pairwise disjoint intervals and a real positive number $a$ such
that $|f'(x)|\geq a>0$ for all
$x\in\T\setminus\bigcup_{j=1}^s(a_j;b_j)$. Then
\begin{equation}
\left|\int_\T e^{2\pi
if(x)}dx\right|\leq\frac{1}{2\pi}\frac{\|f''\|_{\infty}}{a^2}+
\frac{s}{\pi a}+\sum_{j=1}^s(b_j-a_j). \label{3}\end{equation}
\end{lemma}

\begin{proof} Without loss of generality we can assume that
$a_1<b_1<\ldots<a_s<b_s<a_1$. Put $D=\bigcup_{j=1}^s(a_j;b_j)$ and
$a_{s+1}=a_1$. Then
\begin{eqnarray*}
\left|\int_\T e^{2\pi if(x)}\,dx\right|
&\leq&\left|\int_{\T\setminus D}e^{2\pi if(x)}\,dx
\right|+\sum_{j=1}^s(b_j-a_j)\\ & = & \left|\int_{\T\setminus
D}\frac{1}{2\pi if'(x)}\,de^{2\pi if(x)}\right|+
\sum_{j=1}^s(b_j-a_j).
\end{eqnarray*}
Integrating by parts gives
\begin{eqnarray*} \lefteqn{\left|\int_{\T\setminus D}\frac{1}{2\pi
if'(x)}\,de^{2\pi if(x)}\right|}\\
&=&\left|\sum_{j=1}^s\left(\frac{e^{2\pi if(a_{j+1})}}{2\pi
f'(a_{j+1})}- \frac{e^{2\pi if(b_j)}}{2\pi
f'(b_j)}-\frac{1}{2\pi}\int_{b_j}^{a_{j+1}}e^{2\pi
if(x)}\,d\left(\frac{1}{f'(x)}\right)\right)\right|\\
&=&\left|\sum_{j=1}^s\left(\frac{e^{2\pi if(a_{j+1})}}{2\pi
f'(a_{j+1})}- \frac{e^{2\pi if(b_j)}}{2\pi
f'(b_j)}+\frac{1}{2\pi}\int_{b_j}^{a_{j+1}}e^{2\pi
if(x)}\frac{f''(x)}{(f'(x))^2}\,dx\right)\right|\\
& \leq & \frac{1}{2 \pi}\sum_{j=1}^s\left[\left(\frac{1}{|
f'(a_{j})|}+\frac{1}{|f'(b_j)|}\right)+\sum_{j=1}^s|a_{j+1}-b_j|
\frac{\|f''\|_{\infty}}{a^2}\right]\\ & \leq &
\frac{1}{2\pi}\frac{\|f''\|_{\infty}}{a^2}+\frac{s}{\pi a}.
\end{eqnarray*}
\end{proof}

\begin{lemma}\label{lemat}
There exists $C>0$ such that
\[I_m:=\left|\int_{\T^2}e^{2\pi i
(\varphi^{(m)}(x)+\psi^{(m)}(y))}\,dxdy\right|\leq C\frac{\log
u(m)}{u(m)^{1/3}}.\]
\end{lemma}

\begin{proof}  For each $m$ large enough there  exists a natural number
$k\geq n_0$ such that
\[u^{-1}(e^{q_{k}})\leq \frac{m}{u(m)} \leq u^{-1}(e^{q'_{k}})\text{ or }
u^{-1}(e^{q'_{k}})\leq \frac{m}{u(m)} \leq u^{-1}(e^{q_{k+1}}).\]
Suppose that $m/u(m)\in [u^{-1}(e^{q_{k}});u^{-1}(e^{q'_{k}})]$.
Then
\[\frac{m}{u(m)}\leq u^{-1}(e^{q'_{k}})\leq\frac{q_{k+1}}{2u(q_{k+1})}\leq \frac{q_{k+1}/2}{u(q_{k+1}/2)} \]
and hence $m\leq q_{k+1}/2$ because of the concavity of $u$.

 Put
\[a_j=\frac{1}{2q_k}\left(j-\frac{1}{u(m)^{1/3}}\right)-\frac{(m-1)\alpha}{2},\;\;
b_j=\frac{1}{2q_k}\left(j+\frac{1}{u(m)^{1/3}}\right)-\frac{(m-1)\alpha}{2}\]
for $j=1,\ldots,2q_k$. If
$x\in\T\setminus\bigcup_{j=1}^{2q_k}(a_j;b_j)$, then
\[1/u(m)^{1/3}\leq\|2q_k(x+(m-1)\alpha/2)\|\leq |\sin 2\pi
q_k(x+(m-1)\alpha/2)|.\] By (\ref{kocpoch}),
\begin{eqnarray*}
|\varphi'^{(m)}(x)|&\geq&\frac{1}{u^{-1}(e^{q_{k}})}\left|\im
e^{2\pi i q_{k}x} \frac{e^{2\pi imq_{k}\alpha}-1}{e^{2\pi
iq_{k}\alpha}-1}\right|\\&& -
\sum_{n=n_0}^{k-1}\frac{1}{u^{-1}(e^{q_{n}})}\frac{|e^{2\pi
imq_{n}\alpha}-1|}{|e^{2\pi iq_{n}\alpha}-1|}-
\sum_{n=k+1}^{\infty}\frac{1}{u^{-1}(e^{q_{n}})}\frac{|e^{2\pi
imq_{n}\alpha}-1|}{|e^{2\pi iq_{n}\alpha}-1|}.
\end{eqnarray*}
Note that
\begin{eqnarray*}
\lefteqn{\left|\im e^{2\pi i q_{k}x} \frac{e^{2\pi
imq_{k}\alpha}-1}{e^{2\pi iq_{k}\alpha}-1}\right|}\\&=&
\left|\frac{1}{2i}\left(e^{2\pi iq_{k}x}\frac{e^{2\pi
iq_{k}m\alpha}-1}{e^{2\pi iq_{k}\alpha}-1}- e^{-2\pi
iq_{k}x}\frac{e^{-2\pi iq_{k}m\alpha}-1} {e^{-2\pi
iq_{k}\alpha}-1}\right)\right|\\ &=&
\left|\frac{1}{2i}\frac{e^{2\pi iq_{k}m\alpha}-1}{e^{2\pi
iq_{k}\alpha}-1}\left(e^{2\pi iq_{k}x}- e^{-2\pi
iq_{k}(x+(m-1)\alpha)}\right)\right|
\\ &=& \frac{|e^{2\pi
iq_{k}m\alpha}-1|}{|e^{2\pi iq_{k}\alpha}-1|}|\sin 2\pi
q_{k}(x+(m-1)\alpha/2)|.
\end{eqnarray*}
Since $m\leq q_{k+1}/2$ and $\|q_{k}\alpha\|<1/q_{k+1}$, we have
\[\|mq_{k}\alpha\|\leq m\|q_{k}\alpha\|\leq \frac{1}{2}q_{k+1}\|q_{k}\alpha\|<\frac{1}{2},\]
hence $\|mq_{k}\alpha\|= m\|q_{k}\alpha\|$. It follows that
\[\frac{|e^{2\pi
iq_{k}m\alpha}-1|}{|e^{2\pi iq_{k}\alpha}-1|}\geq
\frac{\|q_{k}m\alpha\|}{2\|q_{k}\alpha\|}=\frac{m}{2}.\] Thus
$${\left|\im e^{2\pi i q_{k}x} \frac{e^{2\pi
imq_{k}\alpha}-1}{e^{2\pi iq_{k}\alpha}-1}\right|} \geq
\frac{m}{2u(m)^{1/3}u^{-1}(e^{q_k})}\;.$$

\noindent
 Since
$\|q_{n}\alpha\|>1/(2q_{n+1})$, we have
\begin{eqnarray*}\sum_{n=n_0}^{k-1}\frac{1}{u^{-1}(e^{q_{n}})}\frac{|e^{2\pi
imq_{n}\alpha}-1|}{|e^{2\pi iq_{n}\alpha}-1|}&\leq&
\sum_{n=n_0}^{k-1}\frac{1}{u^{-1}(e^{q_{n}})}\frac{1}{2\|q_{n}\alpha\|}\leq
\sum_{n=n_0}^{k-1}\frac{1}{u^{-1}(e^{q_{n}})}q_{n+1}\\&\leq&
q_{k}\sum_{n=n_0}^{k-1}\frac{1}{u^{-1}(e^{q_{n}})}\leq q_k.
\end{eqnarray*}
Moreover
\[\sum_{n=k+1}^{\infty}\frac{1}{u^{-1}(e^{q_{n}})}\frac{|e^{2\pi
imq_{n}\alpha}-1|}{|e^{2\pi iq_{n}\alpha}-1|}\leq
m\sum_{n=k+1}^{\infty}\frac{2}{u^{-1}(e^{q_{n}})}\leq
\frac{4m}{u^{-1}(e^{q_{k+1}})}.\] Therefore, if
$x\in\T\setminus\bigcup_{j=1}^{2q_k}(a_j;b_j)$, then
\[|\varphi'^{(m)}(x)|\geq\frac{m}{2u(m)^{1/3}u^{-1}(e^{q_k})}-q_k-\frac{4m}{u^{-1}(e^{q_{k+1}})}.\]
Since $u^{-1}(e^{q_k})\leq m/u(m)\leq m$, we have
\[q_k\leq \log u(m)\leq\frac{\log u(m)}{u(m)^{2/3}}\frac{m}{u(m)^{1/3}u^{-1}(e^{q_k})}.\]
Moreover, since
\[u^{-1}(e^{q_k})\leq \frac{m}{u(m)}\leq q_{k+1} \text{ and }u(m)\leq m^{3/4},\]
we have
\begin{eqnarray*}\frac{m}{u^{-1}(e^{q_{k+1}})}&\leq&
\frac{m}{u(m)^{1/3}u^{-1}(e^{q_k})}\frac{m/(u(m))^{2/3}}{u^{-1}(e^{m/u(m)})}\\
&\leq&
\frac{m}{u(m)^{1/3}u^{-1}(e^{q_k})}\frac{(m/u(m))^{2}}{u^{-1}(e^{m/u(m)})}
.\end{eqnarray*} Therefore, for $m$ large enough,
\begin{equation}\label{dolwz}
|\varphi'^{(m)}(x)|\geq \frac{m}{4u(m)^{1/3}u^{-1}(e^{q_k})}\text{
for all } x\in\T\setminus\bigcup_{j=1}^{2q_k}(a_j;b_j).
\end{equation}
On the other hand, by Lemma~\ref{wnpoch},
\[ |\varphi''^{(m)}(x)|\leq\frac{48mq_k}{u^{-1}(e^{q_k})}+q_k\leq \frac{50mq_k}{u^{-1}(e^{q_k})}.\]
An application of Lemma~\ref{lemfund} for the function
$\varphi^{(m)}$ and the family of intervals $(a_i;b_i)$,
$i=1,\ldots,2q_k$ gives
\begin{eqnarray*}\lefteqn{\left|\int_{\T}e^{2\pi
i\varphi^{(m)}(x)}\,dx\right|}\\
&\leq&\frac{1}{2\pi}\frac{\frac{50mq_k}{u^{-1}(e^{q_k})}}{\left(\frac{m}{4u(m)^{1/3}u^{-1}(e^{q_k})}\right)^2}+
\frac{2q_{k}}{
\frac{\pi m}{4u(m)^{1/3}u^{-1}(e^{q_k})}}+\frac{2}{u(m)^{1/3}}\\
&=&\frac{400q_ku^{-1}(e^{q_k})u(m)^{2/3}}{\pi m}
+\frac{4q_ku^{-1}(e^{q_k})u(m)^{1/3}}{\pi
m}+\frac{2}{u(m)^{1/3}}\\&\leq&
\frac{200q_ku^{-1}(e^{q_k})u(m)^{2/3}}{ m} +\frac{2}{u(m)^{1/3}}.
\end{eqnarray*}
Since $u^{-1}(e^{q_k})\leq m/u(m)$, we have $q_k \leq \log u(m)$ and
\[\frac{q_ku^{-1}(e^{q_k})u(m)^{2/3}}{m}\leq \frac{\log u(m)}{u(m)^{1/3}}.\] Consequently
\[\left|\int_{\T}e^{2\pi i\varphi^{(m)}(x)}\,dx\right|\leq 202\frac{\log u(m) }{u(m)^{1/3}}.\]
When $m/u(m)\in [u^{-1}(e^{q'_{k}});u^{-1}(e^{q_{k+1}})]$,
proceeding in the same way we obtain
\[\left|\int_{\T}e^{2\pi i\psi^{(m)}(y)}\,dy\right|\leq 202 \frac{\log u(m) }{u(m)^{1/3}}.\]
Therefore for each natural $m$ we have
\[I_m=\left|\int_{\T}e^{2\pi i\varphi^{(m)}(x)}dx\right| \left|\int_{\T}e^{2\pi
i\psi^{(m)}(y)}dy\right|\leq 202\frac{\log u(m)}{u(m)^{1/3}}.\]
\end{proof}

\begin{proof}[Proof of Theorem~\ref{thm-main-slow}] $\;$

\medskip
\noindent {\bf (i):} Take $f:\T^3\to\R$ given by $f(x,y,z)=\sin 2\pi
z$. Then in view of Lemma~\ref{lemat} we obtain
\begin{eqnarray*}|(f\circ\phi^n,f)|&=&\frac{1}{2}\left|\im\int_{\T^{2}}e^{ 2\pi
i(\varphi^{(n)}(x)+\psi^{(n)}(y))}\,dxdy\right|\\
&\leq&\frac{1}{2}\left|\int_{\T^{2}}e^{ 2\pi
i(\varphi^{(n)}(x)+\psi^{(n)}(y))}\,dxdy\right|\leq
\const\cdot\frac{\log u(n)}{u(n)^{1/3}}
\end{eqnarray*}
for all $n\in\N$ large enough.

\medskip
\noindent {\bf (ii):} Since
\[\phi^n(x,y,z)=(x+n\alpha,y+n\alpha',z+\varphi^{(n)}(x)+\psi^{(n)}(y))\]
it suffices to show that
$\max(\|\varphi'^{(n)}\|_{\infty},\|\psi'^{(n)}\|_{\infty})\leq
c_1u(n)$ for infinitely many $n\in\N$. By Lemma~\ref{nierpod} and
Lemma~\ref{wnpoch},
\[\|\varphi'^{(q_k)}\|_{\infty}\leq 1\text{ and }
\|\psi'^{(q_k)}\|_{\infty}\leq
\frac{6q_k}{u^{-1}(e^{q'_{k-1}})}+q'_{k-1}.\] From
(\ref{zalalpha}) we have
\[q'_{k-1}\leq \log u(q_k)\text{ and }u^{-1}(e^{q'_{k-1}})\geq\frac{q_k}{3u(q_k)}.\]
It follows that
\[\|\psi'^{(q_k)}\|_{\infty}\leq 18u(q_k)+\log u(q_k)\leq 20u(q_k)\]
for all $k$ large enough.

\medskip
\noindent{\bf (iii):} Set
\[g_m:=\max(\|\varphi'^{(m)}\|_{\infty},\|\psi'^{(m)}\|_{\infty})\text{ and
}\widehat{g}_m=\max_{0\leq i\leq m}g_i.\] It suffices to show that
\begin{equation}\label{eq-doubleest}
c_2\frac{\sqrt{m}}{\log u(m)}\leq \widehat{g}_m  \leq
c_3u(\sqrt{m})\sqrt{m}\;,\end{equation} where the left hand side
inequality holds for every natural $m$ and the right hand side holds
for infinitely many $m$.

\medskip

By Lemma~\ref{wnpoch},
\begin{equation}\label{gora}
\widehat{g}_m\leq \max\left(\frac{6m}{u^{-1}(e^{q_{k}})}+q_k,
\frac{6m}{u^{-1}(e^{q'_{k}})}+q'_k\right)
\end{equation}
for every natural $m$ and $k$. Choose $x$ and $y$ so that
$$\sin (2\pi q_k (x+(m-1)\alpha/2))=\sin (2\pi q'_k
(y+(m-1)\alpha'/2))= 1\;.$$ Proceeding along the same lines as in
the proof of Lemma~\ref{lemat} one readily  shows that
\begin{eqnarray}
\label{dol1}u^{-1}(e^{q_{k}})\leq \frac{m}{u(m)} \leq
u^{-1}(e^{q'_{k}})&\Longrightarrow &
g_m\geq |\varphi'^{(m)}(x)|\geq \frac{m}{4u^{-1}(e^{q_{k}})},\\
\label{dol2}u^{-1}(e^{q'_{k}})\leq \frac{m}{u(m)} \leq
u^{-1}(e^{q_{k+1}})&\Longrightarrow & g_m\geq |\psi'^{(m)}(y)|\geq
\frac{m}{4u^{-1}(e^{q'_{k}})}.
\end{eqnarray}

To prove the lower bound on $\widehat{g}_m$ suppose that
$u^{-1}(e^{q_{k}})\leq m/u(m) \leq u^{-1}(e^{q'_{k}})$ (the case
of $u^{-1}(e^{q'_{k}})\leq m/u(m) \leq u^{-1}(e^{q_{k+1}})$ is
treated similarly).

\medskip
\noindent {\em Case 1.} Suppose that $m\leq(u^{-1}(e^{q_{k}}))^2$.
Set $m_0:=[u^{-1}(e^{q_{k}})]$. Then
\[u^{-1}(e^{q_{k}})/2\leq m_0\leq u^{-1}(e^{q_{k}})\leq m\]
and
\[u^{-1}(e^{q'_{k-1}})\leq q_k\leq\frac{e^{q_k/3}}{2}\leq\left(\frac{u^{-1}(e^{q_{k}})}{2}\right)^{1/4}
\leq m_0^{1/4}\leq\frac{m_0}{u(m_0)}\leq u^{-1}(e^{q_{k}}).\]
Therefore in view of (\ref{dol2}), we obtain
\[\widehat{g}_m\geq g_{m_0}\geq \frac{m_0}{4u^{-1}(e^{q'_{k-1}})}\geq \frac{u^{-1}(e^{q_{k}})}{8u^{-1}(e^{q'_{k-1}})}
\geq \frac{u^{-1}(e^{q_{k}})}{4q_k}\geq\frac{\sqrt{m}}{4\log
u(m)}.\]

\medskip
\noindent {\em Case 2.} Suppose that $m\geq(u^{-1}(e^{q_{k}}))^2$.
Then in view of (\ref{dol1}), we obtain
\[\widehat{g}_m\geq g_{m}\geq\frac{m}{4u^{-1}(e^{q_{k}})}\geq\sqrt{m}/4. \]
The desired lower bound on $\widehat{g}_m$ follows.

\medskip

To prove the upper bound on $\widehat{g}_m$ in formula
\eqref{eq-doubleest}
 we take $m=(q'_k)^2$. Then
\[\frac{6m}{u^{-1}(e^{q_{k}})}+q_k=\frac{6(q'_k)^2}{u^{-1}(e^{q_{k}})}+q_k\leq 18u(q'_k)q'_k+q_k\leq 20u(q'_k)q'_k
\leq 20u(\sqrt{m})\sqrt{m}.\] Moreover
\[\frac{6m}{u^{-1}(e^{q'_{k}})}+q'_k=\frac{6(q'_k)^2}{u^{-1}(e^{q'_{k}})}+q'_k\leq
2q'_k=2\sqrt{m}.\] Finally, from (\ref{gora}) we have
\[\widehat{g}_m\leq 20u(\sqrt{m})\sqrt{m}.\] This completes the
proof.
\end{proof}

\section{Growth of the Rudin-Shapiro shift}\label{RS}
In the present section we prove the following result.
\begin{thm} \label{thm-RS} Fix $d >0$.
There exists a bi-Lipschitz homeomorphism $\phi$ of a compact
measure metric space $(X,\rho,\mu)$ with the following properties:
\begin{itemize}
\item[{(i)}] The upper box dimension (see formula (\ref{def-box}) above)
of $(X,\rho)$  equals $d$. Furthermore, for every $\delta >0$ there
exists a $\delta$-net in $X$ containing at most $\text{const} \cdot
\delta^{-d}$ points (see Condition \ref{cond-met} above);
\item[{(ii)}] The homeomorphism $\phi$ mixes a nonzero
Lipschitz function $f: X \to \R$ with zero mean  at the speediest
possible rate, i.e.\ $(f\circ\phi^k,f)_{L_2(X,\mu)}=0$ for all
$k\neq 0$;
\item[{(iii)}] There exist $c_1,c_2 > 0$ so that the growth rate of $\phi$ satisfies
\[ c_1 \cdot n^{1/d}\leq \hG_n (\phi)  \leq c_2\cdot n^{1/d}\]
for all $n \in \N$.
\end{itemize}
\end{thm}

\medskip
\noindent Thus we confirm that the lower bound (\ref{eq-power}) in
Theorem~\ref{exam-conv} is sharp. As we shall explain below, the
homeomorphism $\phi$ can be chosen as the shift associated to the
Rudin-Shapiro sequence.

In what follows we work in the framework of the theory of symbolic
dynamical systems associated to substitutions (see \cite{Q,PF}). Let
us consider a finite alphabet $\mathcal{A}$. Denote by
$\mathcal{A}^*=\bigcup_{n\geq 1}\mathcal{A}^n$ the set of all finite
words over the alphabet $\mathcal{A}$. A {\em substitution} on
$\mathcal{A}$ is a mapping $\zeta:\mathcal{A}\to\mathcal{A}^*$. Any
substitution $\zeta$ induces two maps, also denoted by $\zeta$, one
from $\mathcal{A}^*$ to $\mathcal{A}^*$  and another from
$\mathcal{A}^\N$ to $\mathcal{A}^\N$ by putting
\[\zeta(a_0a_1\ldots a_n)=\zeta(a_0)\zeta(a_1)
\ldots\zeta(a_n)\text{ for every }a_0a_1\ldots a_n\in\mathcal{A}^*,\]
\[\zeta(a_0a_1\ldots a_n\ldots)=\zeta(a_0)\zeta(a_1)\ldots\zeta(a_n)
\ldots\text{ for every }a_0a_1\ldots a_n\ldots\in\mathcal{A}^\N.\]
If there exists a letter $a \in \mathcal{A}$ so that $\zeta(a)$
consists of at least two letters and starts with $a$, the word
$\zeta^{n}(a)$ starts with $\zeta^{n-1}(a)$ and is strictly longer
than $\zeta^{n-1}(a)$. Thus $\zeta^{n}(a)$ converges in the obvious
sense as $n \to \infty$  to an infinite word $v\in\mathcal{A}^\N$
such that $\zeta(v)=v$.

We can associate to the sequence $v$ a topological dynamical system
as follows. Let $\mathcal{L}(v)$ denote the language of the sequence
$v$, i.e.\ the set of all finite words (over the alphabet
$\mathcal{A}$) which occur in $v$. Let $X_v\subset\mathcal{A}^\Z$
stand for the set of all sequences
$x=\{x_n\}_{n\in\Z}\in\mathcal{A}^\Z$ such that $x_nx_{n+1}\ldots
x_{n+k-1}\in\mathcal{L}(v)$ for all $n\in\Z$ and $k\in\N$.
Obviously, $X_v$ is a compact subset of $\mathcal{A}^\Z$ with the
product topology and $X_v$ is invariant under the two-sided
Bernoulli shift $\phi:\mathcal{A}^\Z\to\mathcal{A}^\Z$,
$[\phi(\{x_k\}_{k\in\Z})]_n=x_{n+1}$. Therefore we can consider
$\phi$ as a homeomorphism of $X_v$.

A substitution $\zeta$ is called {\em primitive} if there exists
$k\geq 1$ such that $\zeta^{k}(a)$ contains $b$ for every
$a,b\in\mathcal{A}$. If $\zeta$ is primitive, the space $X=X_v$ does
not depend on the choice of $v$. Furthermore, the corresponding
homeomorphism $\phi: X \to X$ is minimal and uniquely ergodic.
Unique ergodicity of $\phi$ can be deduced from the analogous result
in \cite[Chapter V]{Q} for the one-sided shift: Given two words
$z,w\in \mathcal{L}(v)$, denote by $\Omega_z(w)$ the number of
appearances of $z$ as a sub-word in $w$. Unique ergodicity of the
one-sided shift yields (see \cite[Corollary IV.14]{Q}) existence of
a positive function $\omega: \mathcal{L}(v) \to (0;1]$ so that for
every $z$
\begin{equation}\label{eq-freq}
\frac{\Omega_z(w)}{\text{length}(w)} \to \omega(z)
\;\;\text{uniformly in}\;\; w \;\;\text{as}\;\;\text{length}(w)\to
\infty\;.
\end{equation}
This in turn yields, exactly as in \cite[Corollary IV.14]{Q},
unique ergodicity of the two-sided shift $\phi$.

Let us consider the Rudin--Shapiro sequence $v=\{v_n\}_{n\geq 0}$
over the alphabet $\mathcal{A}=\{-1,+1\}$ which is defined by the
relation
\[v_0=1,\;\;v_{2n}=v_n,\;\;v_{2n+1}=(-1)^nv_n\text{ for any }n\geq 0.\]
It arises from the fixed point $ABACABDB\ldots$ of the primitive
substitution $A \mapsto AB, B \mapsto AC, C \mapsto DB, D \mapsto
DC$ after replacing $A,B$ by $+1$ and $C,D$ by $-1$. As above, we
associate to the sequence $v$ the topological space $X \subset
\mathcal{A}^{\Z}$ and the two-sided shift $\phi: X \to X$.  Notice
that  $\phi$ is uniquely ergodic as a factor of the corresponding
uniquely ergodic substitution system.

\begin{proof}[Proof of Theorem~\ref{thm-RS}] Let $\mu$ be the unique $\phi$-invariant Borel
probability measure on $X$. We shall show that after a suitable
choice of a metric on $X$, the shift $\phi:X \to X$ possesses
properties (i)-(iii) stated in the theorem.

\medskip
\noindent {\sc Choosing the metric:} Fix a concave increasing
function $u: [0;+\infty) \to [0;+\infty)$ such that $u(0)=0$ and
$u(t)\to+\infty$ as $t\to+\infty$. Then define a metric $\rho$ on
$X$ by putting $\rho(x,y)=e^{-u(k(x,y))}$, where $k(x,y)=\min\{|k|:
x_k \neq y_k, k\in\Z\}$ for two distinct sequences $x,y \in X$. Of
course, $\phi$ is a bi-Lipschitz homeomorphism with $\hG_n (\phi)
\leq e^{u(n)}$.

Denote by $\{p_n(v)\}$ the complexity of the sequence $v$, that is
$p_n(v)$ is the number of different words of length $n$ occurring in
$v$. As it was shown in \cite{Al-Sh}, $p_n(v)=8(n-1)$ for every
$n\geq 2$ (in fact, a simpler estimate $n \leq  p_n(v)  \leq
\text{const} \cdot n$ is sufficient for our purposes, see
Propositions 1.1.1 and 5.4.6 in \cite{PF}).

Suppose that $u(t)=d^{-1}\log t$ for all $t$ large enough. Given
$k>2$, put $p=p_{2k-1}(v)$ and consider all possible words
$w^{(1)},\ldots,w^{(p)}$ from $\mathcal{L}(v)$ of length $2k-1$. Fix
arbitrary elements $x^{(i)} \in X$, $i=1,\ldots,p$ so that
$x^{(i)}_{-k+1}x^{(i)}_{-k+2}\ldots x^{(i)}_{k-2}x^{(i)}_{k-1} =
w^{(i)}$. Note that the points $x^{(i)}$ lie at the distance $\geq
e^{-u(k-1)}= (k-1)^{-1/d}$ one from the other. Furthermore, every
point of $X$ lies at the distance $\leq e^{-u(k)}=k^{-1/d}$ from
$x^{(i)}$ for some $i=1,\ldots,p$. Recalling that $p=16k-16$ we
conclude that the upper box dimension of $X$ equals $d$, and
moreover for every $\delta >0$ there exists a $\delta$-net in $X$
containing at most $\text{const} \cdot \delta^{-d}$ points. Thus we
get property (i) in Theorem \ref{thm-RS}.

\medskip
\noindent{\sc Mixing:} Consider a function $f:X\to\R$, $f(x)=x_0$.
Clearly, $f$ is Lipschitz with respect to $\rho$. Let us check that
$(f\circ\phi^k,f)_{L_2(X,\mu)}=0$ for all $k\neq 0$. We prove this
property by combining the unique ergodicity of $\phi$ with the
following fact, see \cite[Proposition 2.2.5]{PF}:
$$\lim_{N\to \infty} \frac{1}{N}\sum_{n=0}^{N-1}v_nv_{n+k} =
0\;\;\forall k \in \N\;.$$ Indeed, there exists a sequence $y\in
X$ such that $y_n=v_n$ for all $n\geq 0$ (see Lemma
\ref{lem-infseq} below). Then
\begin{eqnarray*}
\int_{X}f(\phi^kx)f(x)\,d\mu(x)&=&\lim_{N\to\infty}\frac{1}{N}\sum_{n=0}^{N-1}f(\phi^{k+n}y)f(\phi^ny)\\
&=&\lim_{N\to \infty} \frac{1}{N}\sum_{n=0}^{N-1}v_nv_{n+k} = 0.
\end{eqnarray*}
This proves property (ii) in Theorem \ref{thm-RS}.

\medskip
\noindent{\sc Growth bounds:} The lower bound (\ref{eq-power}) in
Theorem \ref{exam-conv} yields $\hG_n (\phi) \geq \text{const}\cdot
n^{1/d}$. On the other hand
\[\hG_n (\phi) \leq e^{u(n)}= n^{1/d}  ,\]
which yields property (iii) in Theorem \ref{thm-RS}.

\noindent This completes the proof.
\end{proof}

\begin{rem} \label{exam-infdim}  Let us modify the metric $\rho$
defined above by taking the function $u(t)$ to be of an
arbitrarily slow growth. As a result we get an example of a
bi-Lipschitz homeomorphism $\phi$ of a compact metric measure
space $(M,\rho,\mu)$ of {\it infinite} box dimension which mixes a
Lipschitz function $f$ at the speediest possible rate, that is
$(f,f\circ \phi^n)_{L_2}= 0$ for all $n \in \N$, and such that the
growth rate of $\hG_n (\phi)$ is arbitrarily slow. This
illustrates the significance of Condition \ref{cond-met} on the
metric $\rho$ for the validity of the statement of Theorem
\ref{thm-main}.
\end{rem}

\medskip
\noindent We conclude this section with the following lemma which
was used in the proof of Theorem \ref{thm-RS} above.

\begin{lemma}\label{lem-infseq} There exists a sequence $y \in X$ so that $y_n=v_n$
for all $n \geq 0$.
\end{lemma}
\begin{proof}
By \eqref{eq-freq}, for every $n\in \N$ the word $v_0\ldots v_n$
appears infinitely many times as a subword in $v$. Thus we can find
a sequence of words of the form $y^{(n)}= y^{(n)}_{-n}\ldots
y^{(n)}_{-1}v_0\ldots v_n$, $n\in \N$ in the language
$\mathcal{L}(v)$.  Next we choose a collection
$\{\{n_k^l\}_{k\in\N}\}_{l\in\N}$ of increasing sequences of natural
numbers by the following inductive procedure: Since
$\{y^{(n)}_{-1}\}_{n\in\N}$ takes only two values, we can find an
increasing sequence $\{n_k^1\}_{k\in\N}$ such that
$\{y^{(n_k^1)}_{-1}\}_{k\in\N}$ is constant. Assume that the
sequence $\{n_k^l\}_{k\in\N}$ is already chosen. Choose
$\{n_k^{l+1}\}_{k\in\N}$ as a subsequence of $\{n_k^l\}_{k\in\N}$
for which $\{y^{(n_k^{l+1})}_{-l-1}\}_{k\in\N}$ is constant.  Now we
can define the desired sequence $y=\{y_k\}_{k\in\Z}\in X$ by putting
\[y_{-k}=y^{(n_k^k)}_{-k}\text{ for }k>0\text{ and }y_k=v_k\text{ for }k\geq 0.\]
\end{proof}

\section{Appendix:  Kolmogorov-Tihomirov
formula}\label{sec-ap} In this section we prove formula
\eqref{Ko-Ti-ineq}. Cover $A$ by
$n=\mathcal{N}_{\epsilon/(4R)}(A)$ balls $A_1,\ldots,A_n$ of
radius $\epsilon/(8R)$ centered at $a_1,\ldots,a_n\in A$
respectively and cover $Y$ by $m=\mathcal{N}_{\epsilon/4}(Y)$
balls $Y_1,\ldots,Y_m$ of radius $\epsilon/8$ centered at
$y_1,\ldots,y_m$ respectively. Put $I=\{1,\ldots,n\}$,
$J=\{1,\ldots,m\}$. For a map $\sigma: I \to J$ set
$$X_{\sigma} = \{ f \in \mathcal{D}^A_R(Y)\;:\; f(a_i) \in
Y_{\sigma(i)} \;\forall i \in I\}\;.$$ Obviously, $
\mathcal{D}^A_R(Y)$ is covered by $m^n$ sets $X_{\sigma}$. {\it
Warning:} some of these sets might be in fact empty.

Assume that $f,g \in X_{\sigma} \cap \mathcal{D}^A_R(Y)$. Take any
point $a \in A$. Choose $a_i$ so that $\rho_1(a,a_i) <
\epsilon/(8R)$. Then $\rho_2(f(a),f(a_i)) < \epsilon/8$ and
$\rho_2(g(a),g(a_i)) < \epsilon/8$ since the Lipschitz constant of
$f$ and $g$ is $\leq R$. Furthermore,
$\rho_2(f(a_i),y_{\sigma(i)}) < \epsilon/8$ and
$\rho_2(g(a_i),y_{\sigma(i)}) < \epsilon/8$. Thus
$\rho_2(f(a),g(a)) < \epsilon/2$. Since this is true for all
points $a$ in a compact space $A$ we conclude that
$\text{dist}(f,g) < \epsilon/2$. It follows that the set
$X_{\sigma} \cap \mathcal{D}^A_R(Y)$ is either empty, or is fully
contained in a ball of radius $\epsilon/2$ (in the sense of metric
$\text{dist}$) centered at any of its points.

Looking at all $\sigma \in J^I$, we get a covering of $
\mathcal{D}^A_R(Y)$ by at most $m^n$ of metric balls of radius
$\epsilon/2$, as required.

\begin{acknowledgements} We are grateful to A.~Katok for very useful comments on
the first draft of this paper which have led us to an application
of Theorem \ref{thm-recfin} to bi-Lipschitz ergodic theory. We
thank G.~Forni, G.~Lederman and M.~Sodin for useful discussions
and the referee for very helpful remarks and suggestions.
\end{acknowledgements}

\end{document}